\documentclass[11pt]{article}

\usepackage[margin=1in]{geometry}
\usepackage{amsmath,amssymb,amsthm}
\usepackage{microtype}
\usepackage[hidelinks]{hyperref}

\newtheorem{theorem}{Theorem}[section]
\newtheorem{proposition}[theorem]{Proposition}
\newtheorem{lemma}[theorem]{Lemma}
\theoremstyle{definition}
\newtheorem{definition}[theorem]{Definition}

\newcommand{\F}{\mathbb F}
\newcommand{\E}{\mathbb E}
\newcommand{\1}{\mathbf 1}
\newcommand{\twr}{\operatorname{twr}}
\newcommand{\eps}{\varepsilon}

\title{Robust order properties in the HLM construction}
\author{Dean Menezes\\
\small The University of Texas at Austin, Austin, Texas 78712, USA\\
\small \texttt{dean.menezes@utexas.edu}}
\date{}

\begin{document}
\maketitle

\begin{abstract}
Green's arithmetic regularity lemma over $G=\F_2^n$ has tower-type lower bounds. We show that the large order properties in the Hosseini--Lovett--Moshkovitz--Shapira (HLM) construction survive sparse edits. If $d$ is the top block dimension, then changing at most $|G|/16$ points of the top HLM set $A_s$ still leaves a half-graph of height $d/8$. For a suitable choice of the HLM maps, changing at most $\eps|G|$ points of a superlevel set of the averaged witness still leaves a half-graph of height $d/(8\sqrt{s})\ge\twr(s-2)$. The proofs use Sauer's lemma and the binary code formed by the top-level row traces.
\end{abstract}



\section{Introduction}

Green's arithmetic regularity lemma gives a Fourier-uniform decomposition on $\F_2^n$ with codimension depending only on the accuracy parameter \cite{Green}. Hosseini, Lovett, Moshkovitz, and Shapira (HLM) proved that this dependence must have tower growth \cite{HLM}. Terry and Wolf later showed that $k$-stable sets admit much stronger arithmetic regularity and suggested quantifying the order property in the Green--HLM example \cite{TerryWolf}. We show that the HLM construction contains large half-graphs that survive sparse edits.

\begin{definition}
In a finite abelian group $G$, a set $A\subseteq G$ has the \emph{$k$-order property} if there are elements $a_1,\dots,a_k$ and $b_1,\dots,b_k$ in $G$ such that
\[
 a_p+b_q\in A \quad\Longleftrightarrow\quad p\le q.
\]
Equivalently, the relation $a+b\in A$ contains a half-graph of height $k$. The set $A$ is \emph{$k$-stable} if it has no $k$-order property.
\end{definition}

Set
\[
 h(A):=\max\bigl(\{0\}\cup\{k:A\text{ has the $k$-order property}\}\bigr)
\]
and for $0\le\eta\le1$ set
\[
 h_\eta(A):=\min_{|E|\le\eta|G|}h(A\triangle E).
\]
Thus $h_\eta(A)\ge k$ means that every set differing from $A$ on at most $\eta|G|$ points still has the $k$-order property.

Fix $0<\eps\le1/256$ and put
\[
 s:=\left\lfloor\frac1{16\eps}\right\rfloor\ge16.
\]
If $d=d_s$ is the top HLM block dimension, then for the top HLM set $A_s$ and a suitable superlevel set $S$ of the averaged HLM witness, our main bounds (Theorems~\ref{thm:top-robust} and~\ref{thm:threshold}) are
\[
 h_{1/16}(A_s)\ge\frac d8,
 \qquad
 h_\eps(S)\ge\frac{d}{8\sqrt{s}}.
\]
Since $d\ge\twr(s-1)$, both bounds are tower-type. For the first bound, the top-level row traces form a binary code of relative distance greater than $1/4$. For the second, a large binomial fibre realizes every binary trace pattern.

Recent notions of almost stability instead control the number of half-graphs of a fixed height \cite{TerryWolfHigher,Giron}. The quantity $h_\eta$ measures a different phenomenon: how many point edits are needed before all half-graphs of a given height disappear. Bounded VC dimension also gives efficient periodic approximation after sparse edits \cite{AlonFoxZhao}. For other arithmetic-regularity lower bounds built from the Green--HLM construction, see \cite{Gladkova}.

Throughout, $G=\F_2^n$ is additive, with bilinear form and tower function
\[
 \langle u,v\rangle=\sum_j u_jv_j\in\F_2,
 \qquad
 \twr(0)=1,
 \qquad
 \twr(r+1)=2^{\twr(r)}.
\]
For a statement $P$, write $[P]\in\{0,1\}$ for its indicator.

\section{The HLM construction}

Let $H$ be a subspace of $G$ with orthogonal complement
\[
 H^\perp:=\{y\in G:\langle y,h\rangle=0\text{ for every }h\in H\}.
\]
For a function $f:G\to\mathbb R$ and elements $x$ and $t$ in $G$, set
\[
 \widehat{f|_{x+H}}(t)
 :=\frac1{|H|}\sum_{h\in H}f(x+h)(-1)^{\langle t,x+h\rangle}.
\]
The coset $x+H$ is $\eps$-uniform if $|\widehat{f|_{x+H}}(t)|\le\eps$ for every $t\notin H^\perp$. The subspace $H$ is $\eps$-regular if all but an $\eps$-fraction of its cosets are $\eps$-uniform.

Let
\[
 d_1=1,\qquad d_2=2,\qquad d_3=8,
\]
put $D_0=0$ and $D_i=\sum_{j=1}^i d_j$, and define
\[
 d_{i+1}:=2^{D_i-3}\qquad(i\ge3).
\]
Thus $d_4=2^8$ and $d_5=2^{264}$. This sequence matches the one listed in \cite[\S2.1]{HLM}.\footnote{The displayed recurrence in \cite[\S2.1]{HLM} is inconsistent at $i=3$ with the listed sequence and with the identity used later. We follow the listed sequence, for which $d_4=2^8$.} Set $n=D_s$ and write
\[
 x=(x^1,\dots,x^s),
 \qquad x^i\in\F_2^{d_i},
 \qquad x^{<i}:=(x^1,\dots,x^{i-1}).
\]

Choose maps $\xi_i:\F_2^{D_{i-1}}\to\F_2^{d_i}\setminus\{0\}$ as follows. For $i\le3$, the image $\Xi_i:=\xi_i(\F_2^{D_{i-1}})$ is a basis of $\F_2^{d_i}$; for $i>3$, the map $\xi_i$ is injective, $|\Xi_i|=8d_i$, and
\begin{equation}\label{eq:spanning}
 \text{every subset of $\Xi_i$ of size at least $3|\Xi_i|/4$ spans $\F_2^{d_i}$.}
\end{equation}
Such maps exist by the HLM spanning-set lemma. Define
\[
 A_i:=\{x\in G:\langle x^i,\xi_i(x^{<i})\rangle=0\},
 \qquad B_i:=\1_{A_i},
 \qquad f:=\frac1s\sum_{i=1}^s B_i.
\]

\begin{lemma}\label{lem:dimensions}
For every $i\ge1$, we have $d_i\ge\twr(i-1)$. If $d=d_s$, then
\[
 n=d+\log_2d+3,
 \qquad
 d\ge\frac12\log_2|G|.
\]
\end{lemma}

\begin{proof}
The tower bound is immediate for $i\le4$ and then follows from
\[
 d_{i+1}=2^{D_i-3}\ge2^{d_i}.
\]
Also $d=2^{D_{s-1}-3}$, so $D_{s-1}=\log_2d+3$ and hence $n=d+\log_2d+3$. Since $\log_2d+3\le d$ for $d\ge8$, we have $d\ge n/2$.
\end{proof}

\begin{proposition}[HLM]\label{prop:HLMregularity}
The only $\eps$-regular subspace for $f$ is $\{0\}$. Its index is
\[
 |G|=2^n\ge\twr(s).
\]
\end{proposition}

\begin{proof}
The zero subspace $H=\{0\}$ is vacuously $\eps$-regular because $H^\perp=G$.

Let $H\subseteq G$ be a nonzero subspace. Choose the least $i$ for which some $v\in H$ has $v^i\ne0$. Every element of $H$ then vanishes on the earlier blocks. For $g\in G$, let $\gamma_g\in G$ be supported on block $i$, with
\[
 \gamma_g^i=\xi_i(g^{<i}),
\]
and call $g$ good if $\gamma_g\notin H^\perp$. If $i\le3$, then a nonzero functional on $\F_2^{d_i}$ is nonzero on some basis vector, so at least a $1/d_i\ge1/8$ fraction of prefixes are good. If $i>3$, then \eqref{eq:spanning} shows that more than a quarter of the prefixes are good.

Fix a good $g\in G$ and average over vectors $w\in G$ supported on blocks $i+1,\dots,s$. Put
\[
 X_w:=\widehat{f|_{g+w+H}}(\gamma_g).
\]
Then
\[
 \E_wX_w=\frac1{2s}.
\]
Indeed, the terms $B_j$ with $j<i$ have zero coefficient at $\gamma_g$, the term $B_i$ contributes $1/2$, and for $j>i$ averaging first over the fresh $j$th block replaces $B_j$ with the constant $1/2$, which has zero coefficient at $\gamma_g\notin H^\perp$.

Since $\gamma_g\notin H^\perp$ and $0\le f\le1$, we have $|X_w|\le1/2$. If $q$ is the fraction of $w$ with $X_w>\eps$, then
\[
 \frac1{2s}\le\frac q2+(1-q)\eps,
\]
and $\eps\le1/(16s)$ gives $q\ge7/(8s)$.
Adding such a vector $w$ does not change $\gamma_g$, and for uniform $g\in G$ the sum $g+w$ is uniform in $G$. Hence more than
\[
 \frac18\cdot\frac7{8s}=\frac7{64s}>\eps
\]
of the points of $G$ lie in $H$-cosets that are not $\eps$-uniform. Since every $H$-coset has the same size, more than an $\eps$-fraction of the cosets are not $\eps$-uniform. Thus no nonzero $H$ is $\eps$-regular. Finally, Lemma~\ref{lem:dimensions} gives $2^n\ge\twr(s)$.
\end{proof}

\section{A half-graph in the top level}

Put
\[
 d:=d_s,
 \qquad W:=\F_2^{D_{s-1}},
 \qquad U:=\F_2^d.
\]
Then $G=W\times U$ and $|W|=8d$.

\begin{theorem}
The HLM set $A_s$, with $d=d_s$, has the $d$-order property. Hence
\[
 h(A_s)\ge d\ge\twr(s-1).
\]
\end{theorem}

\begin{proof}
By \eqref{eq:spanning}, we may choose a basis $v_1,\dots,v_d$ of $U$ from $\Xi_s$ and prefixes $w_1,\dots,w_d\in W$ with $\xi_s(w_\ell)=v_\ell$. Let $v_1^*,\dots,v_d^*\in U$ be the dual basis, and for $1\le m\le d$ set
\[
 z_m:=\sum_{j=m+1}^d v_j^*.
\]
Then $\langle z_m,v_\ell\rangle=[\ell>m]$. Taking
\[
 a_\ell:=(w_\ell,0),
 \qquad
 b_m:=(0,z_m)
\]
gives $a_\ell+b_m\in A_s$ exactly when $\ell\le m$.
\end{proof}

\section{Robustness of the top level}

For $C\subseteq W\times U$ and $u\in U$, write
\[
 R_C(u):=\{w\in W:(w,u)\in C\}
\]
for the \emph{row trace} of $C$ at $u$.

\begin{lemma}[Sauer]\label{lem:sauer}
Let $N:=|W|$ and $r\ge1$. If $C\subseteq W\times U$ is $r$-stable, then
\[
 \bigl|\{R_C(u):u\in U\}\bigr|
 \le\sum_{j=0}^{r-1}\binom Nj.
\]
Moreover, for $1\le q\le N/2$,
\[
 \sum_{j=0}^q\binom Nj\le\left(\frac{\mathrm eN}{q}\right)^q.
\]
\end{lemma}

\begin{proof}
If $w_1,\dots,w_r\in W$ are shattered by the row traces, then choosing $u_1,\dots,u_r$ in $U$ so that $w_\ell\in R_C(u_m)$ exactly when $\ell\le m$ gives $(w_\ell,0)+(0,u_m)\in C$ exactly when $\ell\le m$. Since $C$ is $r$-stable, no $r$ elements of $W$ are shattered, so the row family has VC dimension at most $r-1$. Sauer's lemma \cite{Sauer} then gives the first bound, and the second is the standard binomial-sum estimate.
\end{proof}

\begin{theorem}\label{thm:top-robust}
For the HLM set $A_s$ with $d=d_s$, we have
\[
 h_{1/16}(A_s)\ge\frac d8\ge\twr(s-2).
\]
\end{theorem}

\begin{proof}
Set $A=A_s$, $N=|W|=8d$, and $r=d/8$. For distinct $u,v\in U$, the traces $R_A(u)$ and $R_A(v)$ differ exactly at those $w\in W$ for which
\[
 \langle u+v,\xi_s(w)\rangle=1.
\]
By \eqref{eq:spanning}, this equality holds for more than $N/4$ values of $w$. Thus the $2^d$ row traces form a binary code of length $N$ and minimum distance greater than $N/4$.

Let $C\subseteq G$ be $r$-stable and let $M$ be its number of distinct row traces. Lemma~\ref{lem:sauer} gives
\[
 M\le\sum_{j=0}^{r-1}\binom Nj
 \le(64\mathrm e)^r
 =2^d\left(\frac{\mathrm e}{4}\right)^r.
\]
Because the original traces are more than $N/4$ apart, a trace of $C$ can be within $N/8$ of at most one of them. Hence
\[
 \frac{|A\triangle C|}{|G|}
 >\frac18\left(1-\frac{M}{2^d}\right)
 \ge\frac18\left(1-\left(\frac{\mathrm e}{4}\right)^r\right)
 >\frac1{16},
\]
since $r\ge3$ and $\mathrm e<3$.
Therefore no set differing from $A$ on at most $|G|/16$ points is $r$-stable, so $h_{1/16}(A)\ge r$. For $t=\twr(s-2)$, we have $t\ge16$, and Lemma~\ref{lem:dimensions} gives $d\ge2^t$. Since $2^t\ge8t$ for $t\ge16$, we obtain $d/8\ge t$.
\end{proof}

The displayed edit-distance estimate also shows that $h_\eta(A_s)\ge d_s/8$ for every fixed $\eta<1/8$ once $s$ is sufficiently large.

\section{A robust superlevel set}

For $w\in W$, let
\[
 c(w):=\sum_{j<s}B_j(w,0).
\]
The variables $B_1(w,0),\dots,B_{s-1}(w,0)$ are independent fair bits for uniform $w$, because at each level the fresh block is paired with a fixed nonzero vector. Hence
\[
 c\sim\operatorname{Binomial}(s-1,1/2).
\]

\begin{lemma}\label{lem:fibre}
For every choice of $\xi_1,\dots,\xi_{s-1}$, there exist a value $c_0$, a set $F_0\subseteq c^{-1}(c_0)$ of size $m$, and an injective map $\xi_s:W\to U\setminus\{0\}$ satisfying the spanning condition \eqref{eq:spanning} such that
\[
 16\mid m,
 \qquad
 \frac{2d}{\sqrt s}\le m\le\frac{3d}4,
\]
and the vectors $\{\xi_s(w):w\in F_0\}$ are linearly independent.
\end{lemma}

\begin{proof}
Let $F=c^{-1}(c_0)$ be a largest fibre. The central-binomial bound gives
\[
 |F|=8d\,2^{-(s-1)}\binom{s-1}{\lfloor(s-1)/2\rfloor}
 \ge\frac{4d}{\sqrt s}.
\]
Since $s\ge16$, we have $3d/\sqrt s\le\min\{|F|,3d/4\}$, and Lemma~\ref{lem:dimensions} gives $d/\sqrt s>16$. Thus the interval $[2d/\sqrt s,3d/\sqrt s]$ contains a multiple $m$ of $16$, and we may choose $F_0\subseteq F$ of size $m$.

Choose the $8d$ values $\xi_s(w)$ independently and uniformly from $U$. The probability of a zero value, a collision, a hyperplane containing at least $6d$ values, or a dependence among the values on $F_0$ is at most
\[
 8d\,2^{-d}
 +\binom{8d}{2}2^{-d}
 +\left(\frac2{\mathrm e}\right)^d
 +2^{m-d}<1.
\]
Here the final inequality uses $m\le3d/4$ and $d\ge256$. The third term follows from Hoeffding's inequality and a union bound over hyperplanes; the fourth follows by exposing the $m$ vectors successively. Some choice therefore avoids all four events. It is injective and nonzero, satisfies \eqref{eq:spanning}, and is independent on $F_0$.
\end{proof}

\begin{theorem}\label{thm:threshold}
For every choice of $\xi_1,\dots,\xi_{s-1}$, there exist an injective map $\xi_s:W\to U\setminus\{0\}$ satisfying the spanning condition \eqref{eq:spanning} and a threshold $\theta$ such that for
\[
 S:=\{x\in G:f(x)\ge\theta\}
\]
we have
\[
 h_\eps(S)\ge\frac d{8\sqrt s}\ge\twr(s-2).
\]
\end{theorem}

\begin{proof}
Take $c_0$, $F_0$, $m$, and $\xi_s$ from Lemma~\ref{lem:fibre} and set $\theta=(c_0+1)/s$. Since $\sum_{j<s}B_j(w,u)=c(w)=c_0$ for $(w,u)\in F_0\times U$, we have
\[
 (w,u)\in S
 \quad\Longleftrightarrow\quad
 \langle u,\xi_s(w)\rangle=0.
\]
Since the $m$ vectors $\xi_s(w)$ for $w\in F_0$ are independent, the traces on $F_0$ run through all $2^m$ binary patterns, each $2^{d-m}$ times.

Let $E\subseteq G$ satisfy $|E|\le\eps|G|$. The edit density on $F_0\times U$ is at most
\[
 \delta:=\frac{8\eps d}{m}
 \le4\eps\sqrt s
 \le\frac1{4\sqrt s}
 \le\frac1{16}.
\]
At least half the elements $u\in U$ have at most $2\delta m\le m/8$ edits on $F_0\times\{u\}$. Since each original trace has $2^{d-m}$ representatives, at least $2^{m-1}$ original traces have such a representative. A Hamming ball of radius $m/8$ has size at most
\[
 V:=\sum_{j=0}^{m/8}\binom mj\le(8\mathrm e)^{m/8},
\]
so $S\triangle E$ has at least $2^{m-1}/V$ distinct traces on $F_0$.

Since
\[
 V(16\mathrm e)^{m/16}
 \le(8\mathrm e)^{m/8}(16\mathrm e)^{m/16}
 =2^m\left(\frac{\mathrm e^3}{64}\right)^{m/16}
 <2^{m-1},
\]
this number of traces exceeds $(16\mathrm e)^{m/16}$. Applying Lemma~\ref{lem:sauer} on $F_0\times U$ shows that $S\triangle E$ is not $(m/16)$-stable, so
\[
 h_\eps(S)\ge\frac m{16}\ge\frac d{8\sqrt s}.
\]

For $t=\twr(s-2)$, Lemma~\ref{lem:dimensions} gives $d\ge2^t$ and $s\le t$. Since $2^t\ge8t^{3/2}$ for $t\ge16$, we obtain $d/(8\sqrt s)\ge t$.
\end{proof}

The choice of $\xi_s$ in Lemma~\ref{lem:fibre} satisfies the HLM conditions, so Proposition~\ref{prop:HLMregularity} still applies to this same function $f$. Theorem~\ref{thm:threshold} concerns a deterministic binary threshold of that witness; it does not assert that the threshold set $S$ itself has the HLM tower lower bound for Green regularity.

By the definition of $h_\eta$, if $C$ is a union of cosets of a subspace $H\subseteq G$ and $|A\triangle C|\le\eta|G|$, then
\[
 [G:H]\ge h_\eta(A).
\]
Indeed, two left vertices in the same $H$-coset have identical membership rows in $C$, whereas the left rows of a half-graph are distinct. Thus every half-graph in $C$ has height at most $[G:H]$.

\section{Questions}

What is the correct order of $h_\eta(A_s)$ for fixed $0<\eta<1/8$? The proof of Theorem~\ref{thm:top-robust} gives a lower bound of order $d_s$ for all sufficiently large $s$, but we do not know a comparable upper bound. Can the factor $\sqrt s$ in Theorem~\ref{thm:threshold} be removed? It enters only because the proof restricts to one value of the binomial variable $c(w)$.

\end{document}